\documentclass[11pt, letterpaper]{amsart}
\usepackage[left=1in,right=1in,bottom=1in,top=1in]{geometry}
\usepackage{graphicx}
\usepackage{amsmath,amsthm,amssymb}
\usepackage{mathrsfs}
\usepackage[all,cmtip]{xy}
\usepackage{cite}
\usepackage{setspace}
\usepackage{mathtools}
\usepackage{subcaption}

\theoremstyle{plain}
\newtheorem{lem}{Lemma}
\newtheorem{thm}[lem]{Theorem}
\newtheorem{prop}[lem]{Proposition}
\newtheorem{cor}[lem]{Corollary}

\theoremstyle{definition}
\newtheorem{defn}[lem]{Definition}

\theoremstyle{remark}

\newtheorem{example}[lem]{Example}

\DeclareMathOperator{\mix}{mix}
\DeclareMathOperator{\diam}{diam}
\DeclareMathOperator{\Gap}{Gap}

\onehalfspace

\begin{document}

\title{A note on sampling graphical Markov models}
\author{Megan Bernstein}\thanks{School of Mathematics, Georgia Institute of Technology.
 Email: bernstein@math.gatech.edu. Supported in part by NSF Grant DMS-$1344199$}
\author{Prasad Tetali}\thanks{School of Mathematics and School of Computer Science, Georgia Institute of Technology. Email: tetali@math.gatech.edu. Supported in part by NSF Grant DMS-1407657.}
\begin{abstract}
We consider sampling and enumeration problems for Markov equivalence classes. We create and analyze a Markov chain for uniform random sampling on the DAGs inside a Markov equivalence class. Though the worst case is exponentially slow mixing, we find a condition on the Markov equivalence class for polynomial time mixing. We also investigate the ratio of Markov equivalence classes to DAGs and a Markov chain of He, Jia, and Yu for random sampling of sparse Markov equivalence classes.

Keywords: graphical Markov model; MCMC algorithm; reversible Markov chain
\end{abstract}

\maketitle

\section{Introduction}

A {\em Bayesian network} or {\em DAG model} is a type of statistical model used to capture a causal relationship in data. The model consists of a directed acyclic graph (DAG) and a set of (dependent) random variables, one variable assigned to each vertex. The DAG encodes conditional independence relations among the random variables. These models are used in areas ranging from computation biology to artificial intelligence~\cite{CompBio,MR1929415,MR2460892,MR2549555}. However, the correct DAG for a system can only be inferred from data up to a condition called {\em Markov equivalence}, where all DAGs in a Markov equivalence class represent the same statistical model~\cite{Heckerman:1995:LBN:218919.218921}. Model selection algorithms face a balance between dealing with the more complicated structure of Markov equivalence classes or encountering inefficiencies and constraints while using DAGs. Works (resulting in partial success) towards understanding Markov equivalence classes through counting and random sampling have considered the questions of enumeration of (1) the Markov equivalences classes on a given number of vertices~\cite{MR1935281,MR3127848,MR3112770,MR1997903}, (2) the DAGs comprising a fixed Markov equivalence class~\cite{MR2253771,2016arXiv161007921H}, or (3) all Markov equivalence classes corresponding to a fixed underlying undirected graph~\cite{2017arXiv170606091R,2016arXiv161107493R}. In spite of much research, the topic of exact enumeration remains stubbornly open. In the present work, we consider the problems of random sampling for questions (1) and (2). 

The graphs in a Markov equivalence class are exactly those that share the same skeleton and immoralities~\cite{MR1935281}: The {\em skeleton} of a directed graph is the underlying undirected graph obtained by removing direction from all the edges. A {\em $v$-structure} (also termed an {\em immorality}) at $c$ occurs among vertices $a,b,c$, whenever the induced subgraph on these vertices has the two directed edges $(a,c)$ and $(b,c)$ but not $(a,b)$ or $(b,a)$. 

An {\em essential graph} is a graphical representation of a Markov equivalence class that utilizes both directed and undirected edges. Since all graphs in a Markov equivalence class share the same skeleton, they only differ in direction of edges. An edge is directed in the essential graph for a Markov equivalence graph if that edge is directed in the same direction in all the graphs in the equivalence class. Otherwise it is undirected. The partially directed graphs (PDAG's) resulted from this are distinct.

A result of Andersson et al~\cite{andersson1997} gives a characterization of the PDAGs that are essential graphs with four conditions \cite{MR3127848}. (1) No partially directed cycles (i.e. chain graph), (2) The subgraph formed by taking only the undirected edges is chordal, (3) The graph in Figure~\ref{av} does not occur as an induced subgraph, and 
\begin{figure}
 \centering
 \includegraphics[scale = .7]{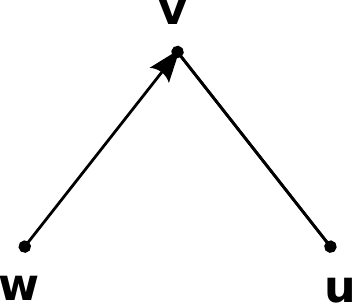}
\caption{Forbidden subgraph}\label{av}
\end{figure}

(4) Every directed edge is strongly protected: A directed edge $u \rightarrow v$ is strongly protected if it occurs in one of the four induced subgraphs in Figure~\ref{pd}.

\begin{figure}
 \centering
 \includegraphics[scale = .7]{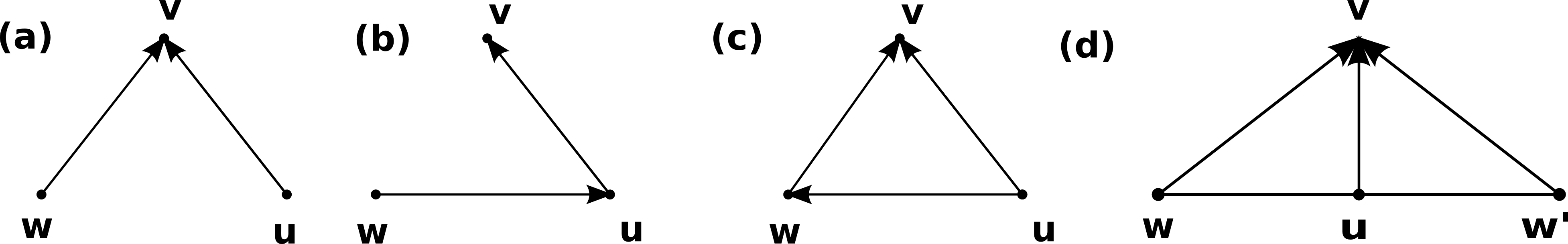}
\caption{Protected edges}\label{pd}
\end{figure}

The essential graphs for Markov equivalence classes containing a single DAG are the essential graphs with only directed edges. In one direction, this follows since if the class has one DAG then all the edges are directed consistently within the class. In the other direction, if there are two DAGs in the class they have the same skeleton and can only differ in the direction of some edge. That edge would then be undirected in the essential graph.

A PDAG with no undirected edges has fewer conditions to be an essential graph: (1) It is a DAG (2) All edges $u \rightarrow v$ are protected by being in one of three induced graphs (a),(b),(c).

\begin{prop}\label{protected}
An edge $u \rightarrow v$ is protected in a PDAG with only directed edges if $\{w| w \rightarrow u\} \neq \{w | w\neq u, w \rightarrow v\}$
\end{prop} 
\begin{proof}
If there exists a $w$ in the first set but not the second, then $u,v,w$ form the induced subgraph in (b). If $w$ is in the second set but not the first, then either $u \rightarrow w$ or not. In the first case this forms the induced subgraph in (c), otherwise the induced subgraph of (a).
\end{proof} 

This paper is comprised of three sections. Section~\ref{mciec} investigates a Markov chain for uniform generation of the DAGs in a Markov equivalence class. It finds a class of graphs on which the associated Markov chain is slow mixing, but a condition for fast mixing as well. The key barrier to fast mixing is large cliques with substantial intersection (roughly half). Section~\ref{paeg} gives a structure theorem for understanding Markov equivalence classes in terms of posets and uses the structure theorem to explore the ratio of DAG's to Markov equivalence classes. Section~\ref{omc} relates the observations to a Markov chain for uniform generation of Markov equivalence classes by He, Jia, and Yu \cite{MR3127848}. In particular, this part provides a simpler, shorter proof of the fact that the chain due to He et al is ergodic. This section concludes with the construction of sparse PDAG's with small Hamming distance but large distance in the chain using moves with positive probability. This highlights the fact that analysis of convergence to equilibrium of the chain is unlikely to be successful using straightforward canonical path arguments or coupling techniques.

The related problem of counting the number of DAGs in a Markov equivalence class has been studied combinatorially in \cite{MR2253771} and algorithmically in recent work of He and Yu\cite{2016arXiv161007921H}. The latter's algorithm is best suited to graphs with many vertices adjacent to all other vertices. Our Markov chain is fast mixing on many graphs that lack this feature, but the simplest example (see Proposition \ref{slow}) of a graph on which our Markov chain is slow mixing is well suited to their algorithm. There are however, many graphs that are ill suited for both, such as chains of half overlapping large cliques.

\section{Edge Flip Random Walk on Chordal Graphs}\label{mciec}

This section investigates the mixing time of a Markov chain designed to pick random samples from the DAGs forming the equivalence class corresponding to a specific essential graph. This involves choosing acyclic orientations for each of the undirected edges in the essential graph in such a way as to form no $v$-configurations. (Recall that a $v$-configuration is an induced subgraph on three vertices depicted in  Figure~\ref{pd}(a).) This depends on only the undirected edges of the essential graph and can be done for each connected and undirected component separately.

The Markov chain is then on $v$-configuration-free acyclic orientations of connected chordal undirected graphs $G = (V,E)$. A step in the Markov chain is to reverse the direction of a single edge so as to give another such orientation. Let $H_G$ be the graph with vertices the orientations for $G$ and edges the transitions of the Markov chain. Note that $H_G$ is a $|E|$-regular graph, and the Markov chain has uniform stationary distribution. When drawing $H_G$, we will suppress self-loops.

\begin{prop}\label{usource}
Any  $v$-configuration-free acyclic orientation of a connected graph $G$ has a unique source.
\end{prop}
\begin{proof}
Let $G$ be a connected graph. 
Let $\tilde{G}$ be a $v$-configuration-free  acyclic orientation of $G$. Suppose there are two sources $v$ and $w$ in $\tilde{G}$. $G$ is connected, so there exists a minimal length undirected path from $v$ to $w$, $(v,v_1,v_2,...,v_k,w)$. Since $v$ and $w$ are sources, the edge between $v$ and $v_1$ is directed $v \rightarrow v_1$ and the edge $w$ to $v_k$ is directed $v_k \leftarrow w$. There must be at least one point on this path where right directed edges meet left directed edges, $v_i \rightarrow v_{i+1} \leftarrow v_{i+2}$. Since $\tilde{G}$ is $v$-configuration-free, there must be an edge $v_i$ to $v_{i+2}$ in $G$. However, this gives a shorter path $v$ to $w$. Therefore, there can only be a single source in $G$.
\end{proof}

As a side note, the chordal condition for the undirected component of a PDAG exists because there are only acyclic $v$-configuration-free orientations of chordal graphs. This means the lack of chordal in the above proposition is not meaningful, as the result is vacuous for non-chordal graphs.

\begin{prop}\label{orientedges}
The unique source of a $v$-configuration-free acyclic orientation of a connected graph $G$ determines the orientation of all edges with one endpoint closer to the source than the other as away from the source.
\end{prop}
\begin{proof}
The proof follows by induction on the distance from the edge to the source. If the edge is incident to the source, i.e. distance zero, then the edge must be oriented away from the source. Suppose this holds for all edges at most distance $d$ from the source. Let $e = \{v,w\}$ be an edge distance $d+1$ from the source with $v$ closer to the source. This means there must be an edge $f = \{v,z\}$ incident to $d$ at distance $k$ whose orientation is forced to be towards $v$. Moreover, there cannot be an edge from $z$ to $w$ since $v$ is closer to the source. If $e$ were oriented towards $v$, then $z \rightarrow v \leftarrow w$ forms a $v$-configuration. Since the orientation is $v$-configuration-free, $e$ is oriented away from the source. 
\end{proof}


\begin{example}
If $G = K_n$, the $v$-configuration-free acyclic orientations are in bijection with the permutations of $n$. This follows from Proposition \ref{orientedges} by successively choosing from the remaining vertices a source which orients the edges adjacent to that source. Only the edges between two consecutive choices of sources can be flipped, so this Markov chain is the adjacent transposition walk on the symmetric group $S_n$. The adjacency graph $H_G$ is the Permutohedron.
\end{example}

\begin{example}
If $G$ is the path, or indeed any tree, $H_G$ is isomorphic with $G$. Since there is a unique path from any vertex to any other vertex, by Proposition \ref{orientedges}, the states are entirely determined by the unique source in the orientation. A move in the Markov chain moves the source to an adjacent vertex.
\end{example}

These two cases are the extreme examples, and the structure for any $G$ can be viewed by decomposing into pieces of these types. The following structure theorem for chordal graphs is the key (see Theorem~12.3.11 in Diestel \cite{MR2744811}):

\begin{prop}
A graph is chordal if and only if it has a tree-decomposition into complete parts.
\end{prop}

\begin{prop}
A tree-decomposition of a chordal graph is a tree $T$ with vertices the maximal cliques of $G$ and for any vertices $t_1,t_2,t_3 \in T$, if $t_2$ is along the unique path from $t_1$ to $t_3$, $t_1 \cap t_3 \subseteq t_2$.
\end{prop}

\begin{figure}
\centering
 \includegraphics[scale=.8]{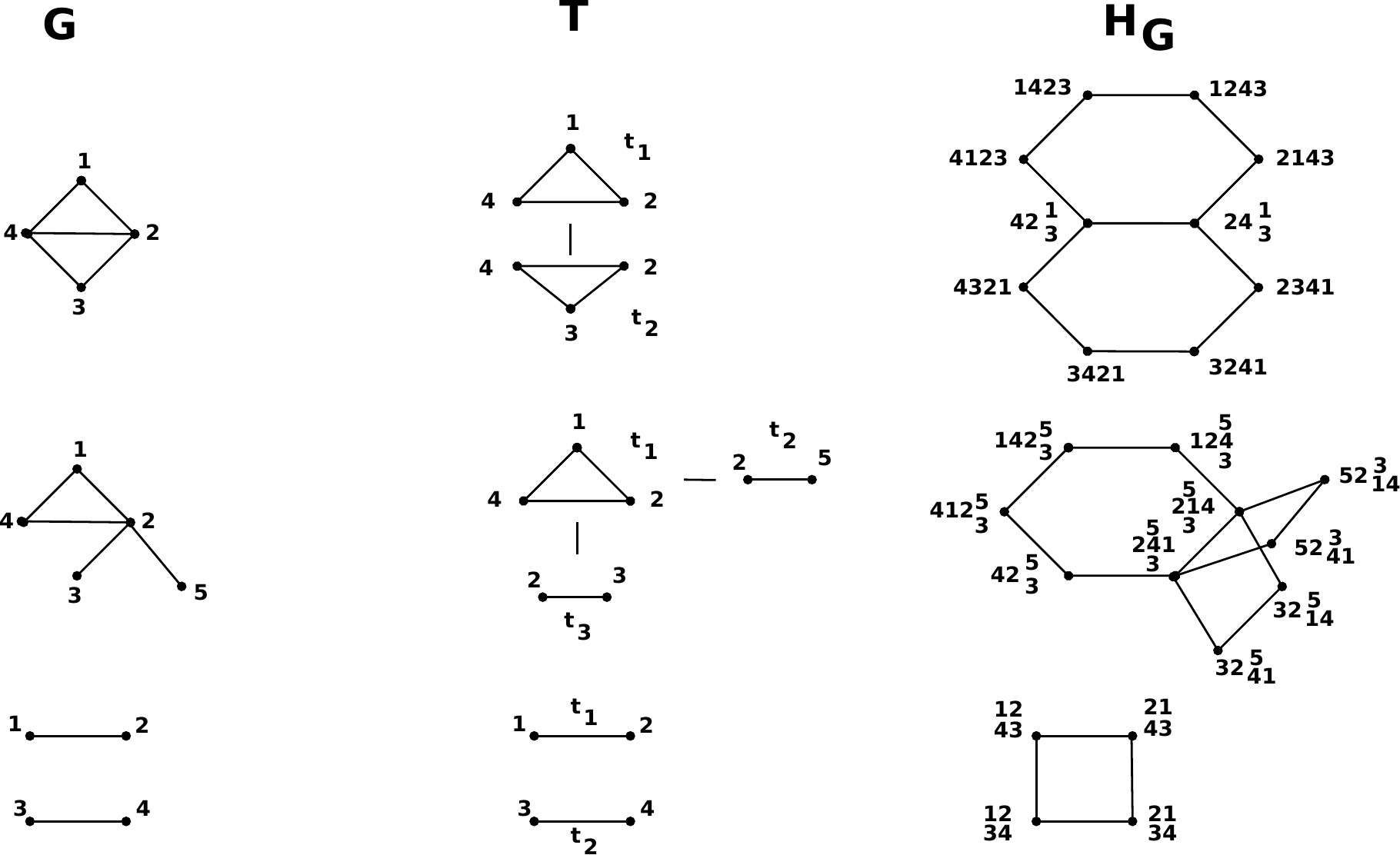}
\caption{Three graphs with corresponding tree decompositions and adjacency graphs}
\end{figure}

\begin{defn}
A maximal clique $t$ is a {\em non-follower} in an orientation of $G$ when for all edges $(v,w)$ with $w \in t$, then $v \in t$ as well. 
\end{defn}

The structure of the graph $H_G$ is closely related to the structure of the tree decomposition of $G$. The following characterization of $H_G$ will be the key to analyzing the v-configuration free edge flip random walk on $G$, or equivalently the random walk on $H_G$. As the v-configuration free edge flip walk on a clique is the random walk on a permutohedron, the key to understanding $H_G$ is to describe how the permutohedra for the maximal cliques in the tree decomposition of $G$ occur in $H_G$. 

To get $H_G$, first for each maximal clique $t_i$, its graph $H_{t_i}$, a permutohedron, is dilated by taking the Cartesian product with other permutohedra. Then these pieces are glued together by identifying faces of the polytopes to form $H_G$.
Given a tree-decomposition $T$ of $G$ and a maximal clique $t_i$, for each other maximal clique $t_j$ let $s_j$ be the clique immediately before $t_j$ on the unique path from $t_i$ to $t_j$. The dilation of $t_i$ will be $D_i = \prod_{j \neq i}  H_{t_j  \cap s_j^c}$ where the product is Cartesian product. Then glue for $t_i$ and $t_j$ adjacent in $T$, $H_{t_i} \times D_i$ and $H_{t_j} \times D_j$ along $H_{t_i \cap t_j} \times D_{i,j}$ where $D_{i,j} = H_{t_i \cap t_j^c} \times D_i = H_{t_j \cap t_i^c} \times D_j$.

\begin{prop}
$H_G$ is formed by first making $H_{t_i} \times D_i$ for each maximal clique $t_i$ in $G$. For each pair of maximal cliques with non-empty intersection, their respective pieces are identified along the faces $H_{t_i \cap t_j} \times D_{i,j}$. 
\end{prop}
\begin{proof}

As shown above in Proposition \ref{usource}, each acyclic, $v$-configuration-free orientation of $G$ has a unique source. This source as in Proposition \ref{orientedges} determines the orientation of all edges with one end point closer to the source than the other. Fixing only a source and orienting these edges breaks down the graph into disjoint components with independence on how to orient the disconnected pieces. This will give rise to the decomposition. The maximal cliques containing the source remain in the same component. To orient the remaining edges, sources are recursively chosen until all maximal cliques are disconnected. The non-followers are the maximal cliques containing all recursive choices of sources up to when the maximal cliques are disconnected. An orientation will be part of $H_{t_i} \times D_i$ when $t_i$ is a non-follower in the orientation. The gluing comes from when multiple cliques are non-followers in the orientation. $H_{t_i \cap t_j} \times D_{i,j}$ are all the orientations from choosing the first $|t_i \cap t_j|$ recursive sources in $t_i \cap t_j$.
\end{proof}


\begin{cor}
Let $C(G)$ be the number of maximal cliques in $G$. For an orientation $v \in H_G$, let $M(v)$ be the number of non-following cliques in $v$. The degree in $H_G$ of an orientation is $|G| - C(G) + M(v)-1$. The minimal degree is $|G| - C(G)$.
\end{cor}
\begin{proof}

The degree of the vertices in a component for a clique $H_{t_i}$ is $|t_i|-1$, the number of adjacent transpositions on $S_{|t_i|}$. Taking cartesian products adds the degree of the graphs involved, so the degree of a vertex in $H_{t_i} \times D_i$ is $|G|-C(G)$. Its left to count how many edges extend outside of $H_{t_i} \times D_i$. These components are glued together along the pairs $t_i,t_j$ with $t_i$ adjacent to $t_j$ in $T$ and $t_i \cap t_j \neq \emptyset$. The number of overlapping edges between the pieces glued together is the degree of a vertex in $H_{t_i \cap t_j} \times D_{i,j}$, namely $|G|- C(G) - 1$. Each gluing thus increases the degree of the orientations involved by one. The number of gluings involving an orientation is one less than the number of non-following cliques.

\end{proof}


\begin{prop}\label{slow}
There exists graphs $G$ for which the edge flip random walk is mixing exponentially slowly. For instance, when $G$ is made up with two cliques of size $\frac{2n}{3}$ sharing $\frac{n}{3}$ vertices, the mixing time $t_{\mix} \geq 4^{n-1}$.
\end{prop}
\begin{proof}
$G$ has two maximal cliques $t_1$ and $t_2$, each of size $\frac{2n}{3}$, with intersection $s$ of size $n/3$. For each maximal clique, its component is $C_i = H_{K_{\frac{2n}{3}}} \times H_{K_{\frac{n}{3}}}$, identified along the face $F = H_{K_{\frac{n}{3}}} \times H_{K_{\frac{n}{3}}} \times H_{K_{\frac{n}{3}}}$, following our construction of $H_G$ above. The face $F$ corresponds to orientations where both maximal cliques are non-followers. Let $R$ be equal to one of these components $C_1/F$ without the intersection, i.e. orientations when one clique is the leader. Let $Q(A,B)$ be the chance of moving from $A$ to $B$ in one step of the random walk starting from the uniform stationary distribution, $\pi$. The bottleneck ratio of the random walk on $H_G$ is:
\[ \Phi_* \leq \Phi(R) = \frac{Q(R,R^C)}{\pi(S)}\,.\]

The Markov chain is the edge flip random walk on $G$, so each state has $|E|$ possible moves. The probability of moving from $R$ to $R^C$ is the total number of edges from $R$ to $R^C$ in $H_G$ over $|E|\cdot|S|$. Each orientation of $G$ in $F$ comes from a recursive choice of sources where the first $n/3$ sources are from $s$, the intersection of the maximal cliques, followed by an independent recursive choice of sources in $t_1\setminus s$ and $t_2\setminus s$. From such an orientation, there is a single edge into $R$ corresponding to flipping the edge between the source chosen last in $s$ and first in $t_1\setminus s$. Therefore, the number of edges from $R$ to $R^C$ is $|F| = (n/3)!^3$. The probability of $S$ under the stationary distribution is $|S| = (2n/3)!(n/3)! - (n/3)!^3$ over the number of orientations of $G$. By inclusion-exclusion, the number of orientations is $|C_1| + |C_2| - |F| = 2(2n/3)!(n/3)! - (n/3)!^3$\,.

 \[ \frac{Q(R,R^C)}{\pi(S)} = \frac{1}{|E|}\frac{(n/3)!^3}{(2n/3)!(n/3)!}\frac{2(2n/3)!(n/3)! - (n/3)!^3}{(2n/3)!(n/3)!-(n/3)!^3}  = \frac{1}{|E|} \frac{1}{{2n/3 \choose n/3}-1}\,.\]

Using Stirling's approximation, ${2n/3 \choose n/3} \approx \frac{1}{\sqrt{\pi n}} 4^n$. 
For $n \geq 4$, $4^n \leq \sqrt{n/\pi} 4^n - n$, and $\Phi(S) \leq 4^{-n}$.

By Theorem~7.3 of \cite{LPW}, we have $t_{\mix} \geq  \frac{1}{4} \Phi_{\*}$. This means for this $G$, $t_{\mix} \geq 4^{n-1}$.
\end{proof}

We will use a decomposition theorem of Madras and Randall (see Theorem~1.1 in \cite{MR1910641}) to get an upper bound on the mixing time. Given a reversible Markov chain $P$ on $\Omega = \bigcup_{i=1}^m A_i$ with stationary distribution $\pi$, they construct two types of Markov chains: for each $i$, a ({\em restriction}) Markov chain restricted to each subset $A_i$, and a ({\em projection}) Markov chain on $m$ states, $a_1, a_2, \ldots , a_m$, with $a_i$ representing the set $A_i$.  The mixing time of these auxiliary chains are used to get a bound on the mixing time for the original chain. Let the Markov chain inside each set be $P_{[A_i]}(x,B) = P(x,B) + 1_{x \in B}P(x, A^c_i)$ for $x \in A_i$ and $B \subset A_i$. To define the chain over the subsets, let $\Theta := \max_{x \in \Omega}|\{i: x \in A_i\}|$. Define the chain over the covering to be:
\[ P_H (a_i,a_j) := \frac{\pi(A_i \cap A_j)}{\Theta \pi[A_i]}\,.\]

\begin{thm}[Madras-Randall]
In the preceeding framework, we have
\[ \Gap(P) \geq \frac{1}{\Theta^2} \Gap(P_H) \left(\min_{i=1,...,m} \Gap(P_{[A_i]})\right).\]
\end{thm}

An upper bound on the mixing time of this random walk can be found through an upper bound on the mixing of the Markov chain for each maximal clique and for a constructed biased random walk on the tree from the tree decomposition. The Markov chains for each maximal clique are the adjacent transposition walk, which is well understood. The random walk on the tree will be studied with comparison techniques.

The only impediment to rapid mixing is a quantity relating to the overlap between cliques along a path in the tree decomposition. Let \[o_G := \frac{\sum_i |t_i|! |D_i|}{\min_{(j,k) \in T} |t_j \cap t_k|! |D_{j,k}|}. \]

\begin{thm}
For a graph $G$ on $n$ vertices, let $o_G$ be defined as above. Let $T$ be a tree decomposition with $\Theta= \deg(T)$ being the maximal number of overlapping maximal cliques in $G$. Let $t_{\max{}} = \max_i{ |t_i|}$. Then the spectral gap of the random walk on $H_G$ satisfies:
$$\Gap(H_G) \geq \left(o_G \Theta^3 (|G|-|T|) \diam(T)\right)^{-1} 2\left(1 - \cos(\pi/t_{\max})\right)\,.$$
\end{thm}
\begin{proof}
Using the decomposition theorem of Madras and Randall \cite{MR1910641}, an upper bound on the mixing time of the random walk on $H_G$ can be obtained by understanding the mixing of a random walk $P_{A_i}$ on each component $H_{t_i} \times D_i$ and the mixing of a Markov chain on the tree from the tree decomposition of $G$ into maximal cliques. 

The random walk on the pieces $H_{t_i} \times D_i$ are Cartesian products of the adjacent transposition walk on the symmetric group. Random walk on the Cartesian product is the product chain on the components. The spectral gap for the product $\tilde{\Gamma} = \Gamma_1 \times ... \times \Gamma_d$ with the chance of moving in the $j$th chain being $w_j$ and the spectral gap of the $j$th chain being $\gamma_j$ is:
\[ \tilde{\gamma} = \min_j w_j \gamma_j.\]
Here the spectral gap of the random walk on $H_{K_i}$ is $\frac{2}{i-1}(1 - \cos(\pi/n))$ by a result of Bacher \cite{MR1283297}. Note that the degree of the vertices in $H_{K_i}$ (or Cartesian products thereof) is the same as the degree of $K_i$ (or Cartesian products thereof), namely $i-1$. Therefore, the chance of making a move in a component of the product chain is the size of that clique minus one over the number of vertices minus the number of cliques.
This gives that $\tilde{\gamma}$ is at most:
\[\Gap(P_{A_i}) = \tilde{\gamma} \leq \frac{2}{|G| - |T|}\left(1 - \cos(\pi/t_{\max})\right).\]

The Markov chain on the tree decomposition into cliques has transition probabilities  constructed as follows. Let $\Theta$ denote the maximum degree of the tree. For two cliques $t_i,t_j$ adjacent in the tree, \[P_T(t_i,t_j) = \frac{\pi ( H_{t_i \cap t_j} \times D_{i,j})}{\Theta \  \pi(H_{t_i} \times D_i)}= \left(\Theta { |t_i| \choose |t_i \cap t_j|}\right)^{-1}.\] It has stationary distribution $\pi(t_i) = |t_i|! |D_i| z^{-1}$ where $z = \sum_i |t_i|! |D_i|$.

The spectral gap of this Markov chain on the maximal cliques will be bounded using comparison to a Markov chain on the complete graph with vertices the maximal cliques with the same stationary distribution. For all $i,j$ let $\tilde{P_T}(t_i,t_j) = \pi(j)$. This Markov chain mixes in one step.

The comparison technique of Diaconis-Saloff-Coste \cite{MR1245303}, is in terms of $A$ below, where $\gamma_{k,l}$ is the unique path in $T$ from $t_k$ to $t_l$:

\[ A = \max_{(t_i,t_j) \in T} \frac{1}{P_T(t_i,t_j)\pi(t_i)} \sum_{k,l: (t_i,t_k) \in \gamma_{t_k,t_l} } |\gamma_{k,l}| \pi(t_k)\pi(t_l).\]

To simplify this, note $|\gamma_{k,l}| \leq \diam(T)$, $\sum_{k} \pi(t_k) = 1$ and hence $\sum_{k,l} |\gamma_{k,j}| \pi(t_k) \pi(t_l) \leq  \diam(T)$. Additionally, $P_T(t_i,t_j)\pi(t_i) = \frac{ |D_{i,j}| |t_i \cap t_j|! }{\left(\sum_i |t_i|! |D_i|\right) \Theta } = \frac{1}{o_G \Theta}$. Therefore,

\[ A \leq o_G \Theta \diam(T).\]

This gives for the largest non-trivial eigenvalue $\beta_1$ of $P_T$, $\beta_1 \leq 1 - \frac{1}{A}$, and \[\Gap(P_T) \geq \frac{1}{A} \geq \left( o_G \Theta \diam(T)\right)^{-1}.\]

The result \cite{MR1910641} gives $\Gap(P) \geq \frac{1}{\Theta^2} \Gap(P_H) \min_i \Gap(P_{A_i}) $, which from the bounds above gives:

\[ Gap(P) \geq (o_G \Theta^3 (|G|-|T|) \diam(T))^{-1} 2\left(1 - \cos(\pi/t_{\max})\right). \]

\end{proof}


\begin{cor}
When all the maximal cliques in $G$ are the same size $t$ and all intersecting cliques intersect along $s$ vertices, then $o_G = |T| {t \choose s }$ and the spectral gap satisfies:

\[Gap(P) \geq |T|(|G|-|T|)\Theta^3 {t \choose s} 2\left(1 - \cos(\pi/t)\right)\,. \]
\end{cor}

The extreme cases of $G$ being a complete graph or a tree fall into the purview of this corollary, and the bound on the spectral gap is within a constant of what it should be.

\section{Posets and Essential Graphs}\label{paeg}

A poset, or partially ordered set, is a combinatorial object defined on a set $X$ with a set of relations $P$.  The relations form a partial order in that they are reflexive (for all $x \in X$, $x \leq x$), antisymmetric (not both $x \leq y$ and $y \leq x$ if $y \neq x$), and transitive (if $x \leq y, y\leq z$, then $x \leq z$). A few terms from posets are useful in the looking at essential graphs. We say $y$ covers $x$ if $x < y$ with no $z$ so that $x<z<y$. A chain in a poset is a set $C \subseteq X$ with all elements of $C$ totally ordered by the poset. The height of a point in a poset is the size of the largest chain with that point as the highest point in that total order.

Every DAG can be reduced to a poset by establishing the relations $v < w$ if $w$ is reachable from $v$. For each labeled poset $P$, it is straightforward to count the number of essential graphs with only directed edges that reduce to it as well as the number of DAGs. We recall here that a relation $x<y$ in a poset $P$ is called a {\em cover} relation, if there is no $z$ such that $x < z< y$ in $P$.
\begin{prop}
The number of DAGs is
\[ \sum_{P} \prod_{v \in P} (2^{d(v) - c(v)}),\]
where the sum is over labeled posets, $d(v)$ is the number of poset elements below $v$ in the order, and $c(v)$ is the number of elements covered by $v$.
\end{prop}
\begin{proof}
One can construct the DAGs that have a specific reachability poset by looking at the down set and cover relations of each point of the poset. The points covered by a point $v$ are obliged to have directed edges to $v$. All other elements of the descent set have the option to have a directed edge to $v$. 
\end{proof}

Surprisingly, satisfying the three conditions to protect the edges is straightforward in this setting.

\begin{prop}
The number of Markov equivalence classes of size one is:
\[ \sum_{P} \prod_{v \in P} (2^{d(v) - c(v)} - 1_{c(v) = 1}),\]
where the sum is over labeled posets, $d(v)$ is the size of the down set of $v$, and $c(v)$ is the number of elements covered by $v$.
\end{prop}
\begin{proof}
From a labeled poset $P$, we will count all the essential graphs with only essential edges that reduce to it. The DAG must include all the cover relations in $P$. For each point $v$ in the poset, a choice can be made to include or not a directed arrow to $v$ from any $u$ in the down set of $v$ that is not covered by $v$. If $d(v)$ is the size of the down set of $v$, and $c(v)$ is the number of elements covered by $v$, this gives $2^{d(v)-c(v)}$ ways to have edges come into $v$. These edges are protected as long as the condition $\{w| w \rightarrow u\} \neq \{w | w\neq u, w \rightarrow v\}$ is met. If $v$ covers at least two other points, then no other vertex coming from the down set of $v$ can have an incoming edge from both these points, since they must be incomparable in the poset. If $v$ only covers a single point, then there is a unique way for $\{w| w \rightarrow u\} \neq \{w | w\neq u, w \rightarrow v\}$ to fail: If $u$ is the point that $v$ covers and they have edges from exactly the same vertices. This gives $2^{d(v) - c(v)} - 1_{c(v) = 1}$ ways to have protected directed edges coming into $v$.
\end{proof}

\subsection{Exact enumeration and discussion}

The ratio of essential graphs to DAGs and essential graphs with only directed edges to DAGs is of interest to determine the limit of increased efficiency in working with essential graphs versus DAGs. Using essentially the same observation above, Steinsky~\cite{MR1997903} uses inclusion-exclusion to get a recursive formula for $a_n$, the number of essential graphs with only directed edges, otherwise known as essential DAGs. The inclusion-exclusion works to add a set of $i$ maximal vertices and connect $n − i$ lower vertices to them arriving at the formula, \[a_n = \sum_{i=1}^n (−1)^{i+1}{n \choose i}\left(2^{n-i}-(n-i)\right)^i a_{n-i}\,.\] This is done in the style of Robinson's recursive formula \cite{Robinson1977} for the number of DAGs,\[a_n' = \sum_{i=1}^n (−1)^{i+1}{n \choose i}2^{i(n-i)}a_{n-i}'\,.\] Steinsky has a second paper published in 2013 on enumerating Markov equivalence classes\cite{MR3112770}.

While they are both exceedingly useful for computation, the alternating quality of both formulas makes asymptotic analysis a challenge.  While the poset construction of a formula for $a_n$ would be next to useless for computing $a_n$ for large n, its all positive structure makes it more amenable to inform the ratio of $a_n'/a_n$.

Steinsky computed that the ratio for $n \leq 300$. By $n = 200$, the first 45 decimal places appear to have stabilized, so $a_n'/a_n = 13.6517978587767...$ 

The $q$-Pochammer symbol appears to be playing a role in this constant, as shown below by looking at certain families of posets. The ratio of the $q$-Pochammer symbol that appears and Steinsky's approximations, $a_n'/a_n \left(\frac{1}{2},\frac{1}{2}\right)_{n-2}$, is just under $4$ at $3.94...$

The largest contribution of an unlabeled poset to the formulas above is for the total orders, where there are $n!$ ways to label the poset. This family also has the largest number of DAGs per reachability poset at $\prod_{i=2}^n 2^{i-2} = 2^{n-1 \choose 2}$, since this maximizes the down set and minimizes the covered points. However, none of these DAGs form an essential DAG, since none of these edges is protected. One of the next largest contributions to the DAG formula is from the unlabeled poset that form a total order aside from two incomparable elements at the bottom of a linear order. There are $n!/2$ ways to label this poset and $\prod_{i=4}^n 2^{i-2}$ ways to construct a DAG from it. Now, a large number of these are essential graphs. Specifically,  $\prod_{i=4}^n (2^{i-2}-1)$ of them. This proportion is:

\[ \prod_{i=4}^n \frac{ 2^{i-2} - 1}{2^{i-2}} = \prod_{i=4}^n \left( 1 - \left(\frac{1}{2}\right)^i\right) = 2\left(\frac{1}{2},\frac{1}{2}\right)_{n-2},\]

where $(a,q)_n = \prod_{i=0}^{n-1}(1-aq^i)$ is the $q$-Pochammer symbol. 

This ratio is quite a bit off from the ratio of essential graphs with only directed edges to DAGs, but by combining the two classes together, one gets a lot closer.

There are four times as many DAGs from linear orders as the DAGs from our almost linear orders. This gives instead of $2\left(\frac{1}{2},\frac{1}{2}\right)_{n-2}$ as a ratio, $\frac{2}{5}\left(\frac{1}{2},\frac{1}{2}\right)_{n-2}$. Steinsky's computations would match up with a leading fraction of $\approx \frac{1}{3.94}$.

There should be a natural way to link up posets that lead to essential graphs and those that do not that results in ratios that closely approximate the real ratio.  For instance, all essential graphs on only directed edges have in their reachability poset at least two points of height $1$ covered by each point of height $2$ (and all points of height $2$ cover at least two points height $1$). By taking any pair of these height $1$ points and covering one with the other, one arrives at a poset not giving an essential graph. Moreover, for a poset with height $2$ points not covering at least two height $1$ points, there is a clear map to essential graph posets - make those two points incomparable. Moreover, since the $q$-Pochammer symbol that appears will change as the posets get wider, the constant may be easier to achieve than it first appears. 


To extend this work to understanding the ratio of DAGs to all essential graphs, one needs to understand how undirected edges can be added as well as the directed edges. The undirected edges are not as independently placeable as the directed edges due to the restriction that they must form a chordal graph. Work has been done on counting the size of the equivalence class given a fixed set of undirected edges in \cite{MR2253771, 2016arXiv161007921H } and sampling inside the equivalence class in Section~\ref{mciec}. This is an attractive question in that it depends only on the undirected edges present and not on any directed edges. It may also be beneficial to study how the undirected edges can be placed on top of the directed edges in the poset model.

Undirected edges can be added to an essential graph with only directed edges as follows. Edges are only allowed between vertices $u,v$ with directed edges coming in from exactly the same other vertices. This prevents the disallowed induced subgraphs of Figure~\ref{av}. Care must also be taken to keep the edges coming out of these vertices protected. Using protected edge condition Figure \ref{pd} (d), the edges coming into $v$ are protected after the addition of undirected edges if at least two of the vertices with edges to $v$ have no undirected edge between them. On top of these protection conditions, the undirected edges must also be chordal. It seems like one of the alternate definitions of chordal might be more tractable. For instance, this definition seems more naturally suited to recursive constructions: a graph is chordal if its vertices are partitionable into three sets $A,C,B$ with the induced subgraph on $C$ complete and no edges with one end in $A$ and the other in $B$. 


\section{A Markov Chain on Essential Graphs}\label{omc}

Essential graphs are difficult to count or generate exactly beyond about $20$ vertices. He, Jia, and Yu~\cite{MR3127848} created a reversible Markov chain on essential graphs that can be used to sample uniformly from the essential graphs on a large number of vertices. It is designed to function under a sparsity condition of at most constant density. In order to produce a non-lazy chain, the authors choose to only allow moves that give a different essential graph. This has the effect of giving a non-uniform stationary distribution and making the chain harder to describe. Here, a lazy version of their chain is used in order to make formulas more explicit.

He, Jia, and Yu ultilize $6$ moves  on the state space of essential graphs on $n$ vertices. The first four are changes in one edge. A random pair of vertices are selected uniformly at random. Then a choice is made to attempt to add a directed edge, remove a directed edge, add an undirected edge, or remove an undirected edge. Adding an edge is only allowed if an edge is not already present. Removing an edge is similarly only allowed if that type of edge is already present. Furthermore, the resulting graph need not be an essential graph itself, but must be extendable by directing edges to give a DAG. The graph is modified to be the essential graph of that DAG. A move to add an edge is furthermore only permissible if that edge remains the type added in the essential graph. These additional factors of correcting the graph to be essential are done in order to make the chain reversible.  The final two moves are to pick three vertices $a,b,c$ and add or remove an immoral. An immoral can only be added if there are undirected edges $a-b$ and $b-c$ with $a$ and $c$ not adjacent. As before, the graph is then corrected to be an essential graph if possible or the move is rejected. To remove an immoral at $b$, there must be directed edges $a \rightarrow b$ and $c \rightarrow b$ with no edge from $a$ to $c$. This is changed by the move to undirect the edges, again with the restriction that it repairs to an essential graph with those edges still undirected.


He, Jia, and Yu show the chain is connected by showing an iterative procedure of removing all directed edges followed by removing some directed edge or a $v$-configuration, thus prescribing a path in the chain from any essential graph to the empty graph. The poset relationship described above can be used to give an explicit order in which edges can be removed in order to get to the empty graph. Such a procedure could be of use in a mixing time analysis of the chain.

\begin{itemize}
\item Remove all undirected edges. 
\item Starting with the maximal elements at height $\geq 2$ of the poset for the essential graph, remove all directed edges into the maximal elements one at a time in an order that keeps them from having the same incoming edges as their children. If a maximal element has multiple children this can be done by removing the edges from the children last (removing any children at height $1$ first). If there is a single child, remove all incoming edges in common with that child first. Recursively continue with the new maximal elements.
\item Continue this until only immorals consisting of elements at height $1$ and $2$ are left. Remove these by first removing all possible directed edges as singletons, then for each immoral, turning the immorals into undirected edges and removing those. 
\end{itemize}

\begin{prop}
Every move in the above procedure is a move in the Markov chain proposed by He et al~\cite{MR3127848}.
\end{prop}
\begin{proof}
Of the four conditions for an essential graph, removing undirected edges could only violate the need for the undirected edges to form a chordal graph. Several of the alternate definitions of chordal graphs give that there is an order remove all the undirected edges leaving it chordal at each step. For instance, one definition of chordal is that every chordal graph can be broken up into three sets $A,C,B$ with $C$ non-empty and complete, no edges between $A$ and $B$, and $A$ and $B$ chordal. If $A$ and $B$ are independent, removing edges $A$ to $C$ and $B$ to $C$, followed by the edges adjacent to each vertex in $C$ leaves the graph chordal at each step. Recursively removing the edges from $A$ and $B$, then $A$,$B$ to $C$, then inside of $C$, then gives a way to remove all edges leaving the graph chordal at each step.

For an essential graph with only directed edges, by Proposition \ref{protected} an edge $u \rightarrow v$ is protected if $\{c | c\rightarrow u\} \neq \{c | c\neq u, c \rightarrow v\}$. For a maximal element of the poset $v$, there are no edges coming out of $v$, so removing edges coming into $v$ does not endanger the protectedness of any edges going into other vertices. Its enough to ensure that the edges coming into $v$ are protected at each step of their removal. Let $w_1,...w_r$ be the vertices covered by $v$ in the poset. Each of these is in the set $\{c | c\neq u, c \rightarrow v\}$ when they aren't $u$ but are never in $\{c | c\rightarrow u\} $. As long as there are at least two vertices covered by $v$, all edges into $v$ are protected and all but those two can be removed in any order. Leave one of the edges from the highest vertex, $w_i$, $v$ covers. Since $v$ is height at least $3$, $w_i$ is height at least $2$, and $\{c | c\rightarrow w_i \} \neq \emptyset$. Therefore the edge $w_i \rightarrow v$ is protected if it is the last edge to $v$ left. Remove the other edge, then this edge. 

Suppose instead that $v$ covers a single vertex $w$, with height of $v$ at least $3$. Since the edge $w \rightarrow v$ is protected, $\{c | c\rightarrow w\} \neq \{c | c\neq w, c \rightarrow v\}$. The presence of a vertex in the latter set not in the former would mean $v$ covered at least two vertices. Then the first set contains at least one vertex $u$ not in the other set. Remove all edges to $v$ other than $w \rightarrow v$ in any order, as this edge will remain protected by the presence of $u$ and the presence of $w \rightarrow v$ protects them while they exist. Finally, remove $w \rightarrow v$.

After recursively removing all maximal vertices of height at least $3$, the essential graph left has only vertices of height $1$ with directed edges leading to vertices of height $2$. Each vertex of height two has at least two incoming edges since forming an immoral is the only way such edges can be protected. Prune each immoral down to two edges by removing single edges. Take one immoral $a \rightarrow b \leftarrow c$ with no other edges going to $b$. Then turning it into $a - b - c$ forms an essential graph because no $u \rightarrow v - w$ edges were formed, and all other directed edges remain protected in an immoral. Then these undirected edges can be removed since they are the only undirected edges in the graph. Repeat with the other immorals.
\end{proof}

Define the Hamming distance between two essential graphs to be the number of edges including direction, in which the graphs differ. We can show that using the basic four moves, at most two moves are needed to move between any two essential graphs at Hamming distance one. The chain is not connected by just these moves. The Hamming distance between two essential graphs that are adding or removing an immoral (without changing the direction of other edges), $a \rightarrow b \leftarrow c$ versus $a - b - c$, however suffice by the above procedure. Note, this should also mean the chain would be connected and reversible, if the ``repairing'' moves that do not give essential graphs were left out. 

\begin{prop}
There is a path in the chain of length at most $2$ between any two essential graphs at Hamming distance one. 
\end{prop} 
\begin{proof}
The first statement breaks into three cases. Let $a,b$ be the two vertices at which the graphs differ. Case I is that one graph has no edge $a$ to $b$ while the other has a directed edge, WLOG, $a \rightarrow b$. Case II is that one graph has no edge and the other has an undirected edge $a - b$. Case III is that one graph has a directed edge $a \rightarrow b$ and the other $b \rightarrow a$. It is not possible to differ at a directed versus undirected edge, since by definition an edge in an essential graph is undirected if there is a DAG in which that arrow goes in either direction.

Case I: Since both no edge and the directed edge form essential graphs, the Markov chain moves of add a directed edge and removing a directed edge are legal, and this gives a one step path in either direction under the chain.

Case II: Similarly to Case I, since both no edge and the undirected edge form essential graphs, the Markov chain moves of adding an undirected edge and removing an undirected edge give a one step path between the two essential graphs.

Case III: We will show removing the directed edge and adding it back in the other direction gives a two step path in the chain. This case has the complexity since it is necessary to show that the intermediate move gives an essential graph only differing at that edge. The graph $G$ obtained by removing the directed edge must still be chordal and contain no partially directed cycles since no undirected edges were changed and removing edges can not add a cycle. The third criterion for an essential graph, no $x \rightarrow y - z$, could also not be introduced by removing an edge. It is left to show all the directed edges remain protected. Suppose $a \rightarrow b$ is used as one of the edges in one of the four induced subgraphs that protects that edge $u \rightarrow v$. The fourth induced subgraph cannot be the relevant one, since switching the direction of any of the directed edges gives a partially directed cycle. That means $u \rightarrow v$ is protected by of the first three conditions which by, Proposition \ref{protected}, are equivalent to $\{c| c\rightarrow u\} \neq \{c| c\neq u, c \rightarrow v\}$. The edges $a$ to $b$ are only relevant if one of $u$ or $v$ is $a$ or $b$. Without loss of generality we will check $u=a$ and $v=a$. Suppose $u =a$ and $v \neq b$. In the graph with $a \rightarrow b$, the edge is not in either set and the sets are still not equal. That means with no edge $a$ to $b$ the sets are disjoint and the edge is protected. Suppose $v=a$ and $u \neq b$. In the graph with $b \rightarrow a$, the edge is not counted in either set and the sets are still not equal. That means without the edge $a$ to $b$, the edge $u$ to $v$ is still protected. 

\end{proof}

Unfortunately, the Markov chain does not give short paths between all essential graphs with Hamming distance two. For instance, the graph in figure \ref{coune}, where $I_n$ stands in for an independent graph on $n$ vertices.

\begin{figure}
	\centering
		\includegraphics{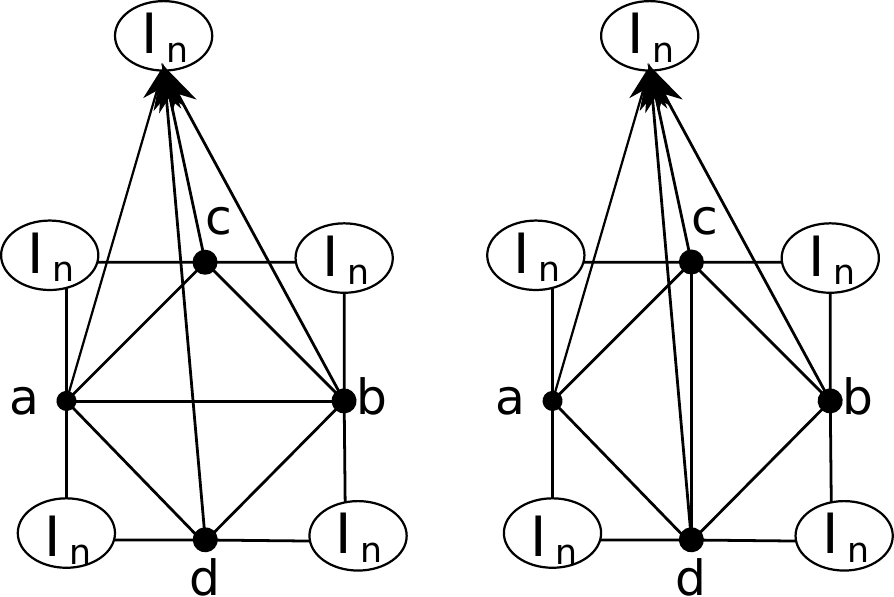}
		\caption{Hamming distance two counter example}\label{coune}
\end{figure}

First note, neither the graph with both $a-b$ and $c-d$ and neither $a-b$ not $c-d$ is an essential graph. The first because in order to protect the directed edges, there must be two non-adjacent vertices. The second, because then the undirected edges would have a four-cycle and fail to be chordal. There are two general approaches to get around this, either add more directed edges going up to the independent graph or delete/add undirected edges. In order to protect all the directed edges with both $a-b$ and $c-d$ present, an extra directed edge would have to be added going to each for the vertices in the independent set. This requires at least $n$ edges. Alternatively, one can try to avoid breaking the chordal condition while manipulating undirected edges. This could be done by either deleting an edge in the cycle formed by $a,c,b,d$ or adding extra chords in the cycles formed between the independent graphs and $a,b,c,d$. In order to delete the edge $a-c$, one first has to delete $n$ edges from $a$ or $c$ to the independent graph to avoid forming a four-cycle. This takes order $n$ moves. In adding chords to avoid making a cycle, a chord has to be added to each vertex in an independent graph, so again order $n$ new edges must be added. Together, this means there is no path between these graphs in $o(n)$ moves of this Markov chain.

Moreover, this example has $5n+4$ vertices and $12n + 5$ edges, well inside the sparsity condition He, Jia, and Yu consider. Namely, that the number of edges being at most a small constant multiple of the number of vertices.

{\em Acknowledgment}. We thank Caroline Uhler and Liam Solus for several helpful discussions on the topic of enumeration of Markov equivalence graphs.

\bibliography{EGWN2}{}

\begin{thebibliography}{10}

\bibitem{andersson1997}
Steen~A. Andersson, David Madigan, and Michael~D. Perlman.
\newblock A characterization of markov equivalence classes for acyclic
  digraphs.
\newblock {\em Ann. Statist.}, 25(2):505--541, 04 1997.

\bibitem{MR1283297}
Roland Bacher.
\newblock Valeur propre minimale du laplacien de {C}oxeter pour le groupe
  sym\'etrique.
\newblock {\em J. Algebra}, 167(2):460--472, 1994.

\bibitem{MR1929415}
David~Maxwell Chickering.
\newblock Learning equivalence classes of {B}ayesian-network structures.
\newblock {\em J. Mach. Learn. Res.}, 2(3):445--498, 2002.

\bibitem{MR1245303}
Persi Diaconis and Laurent Saloff-Coste.
\newblock Comparison techniques for random walk on finite groups.
\newblock {\em Ann. Probab.}, 21(4):2131--2156, 1993.

\bibitem{MR2744811}
Reinhard Diestel.
\newblock {\em Graph theory}, volume 173 of {\em Graduate Texts in
  Mathematics}.
\newblock Springer, Heidelberg, fourth edition, 2010.

\bibitem{CompBio}
Nir Friedman, Michal Linial, Iftach Nachman, , and Dana Pe'er.
\newblock Using bayesian networks to analyze expression data.
\newblock {\em Journal of Computational Biology}, 7:601--620, 2004.

\bibitem{MR2253771}
Steven~B. Gillispie.
\newblock Formulas for counting acyclic digraph {M}arkov equivalence classes.
\newblock {\em J. Statist. Plann. Inference}, 136(4):1410--1432, 2006.

\bibitem{MR1935281}
Steven~B. Gillispie and Michael~D. Perlman.
\newblock The size distribution for {M}arkov equivalence classes of acyclic
  digraph models.
\newblock {\em Artificial Intelligence}, 141(1-2):137--155, 2002.

\bibitem{2016arXiv161007921H}
Y.~{He} and B.~{Yu}.
\newblock {Formulas for Counting the Sizes of Markov Equivalence Classes of
  Directed Acyclic Graphs}.
\newblock {\em ArXiv e-prints}, October 2016.

\bibitem{MR2460892}
Yang-Bo He and Zhi Geng.
\newblock Active learning of causal networks with intervention experiments and
  optimal designs.
\newblock {\em J. Mach. Learn. Res.}, 9:2523--2547, 2008.

\bibitem{MR3127848}
Yangbo He, Jinzhu Jia, and Bin Yu.
\newblock Reversible {MCMC} on {M}arkov equivalence classes of sparse directed
  acyclic graphs.
\newblock {\em Ann. Statist.}, 41(4):1742--1779, 2013.

\bibitem{Heckerman:1995:LBN:218919.218921}
David Heckerman, Dan Geiger, and David~M. Chickering.
\newblock Learning bayesian networks: The combination of knowledge and
  statistical data.
\newblock {\em Mach. Learn.}, 20(3):197--243, September 1995.

\bibitem{LPW}
David~A. Levin, Yuval Peres, and Elizabeth~L. Wilmer.
\newblock {\em Markov chains and mixing times}.
\newblock American Mathematical Society, Providence, RI, 2009.
\newblock With a chapter by James G. Propp and David B. Wilson.

\bibitem{MR2549555}
Marloes~H. Maathuis, Markus Kalisch, and Peter B\"uhlmann.
\newblock Estimating high-dimensional intervention effects from observational
  data.
\newblock {\em Ann. Statist.}, 37(6A):3133--3164, 2009.

\bibitem{MR1910641}
Neal Madras and Dana Randall.
\newblock Markov chain decomposition for convergence rate analysis.
\newblock {\em Ann. Appl. Probab.}, 12(2):581--606, 2002.

\bibitem{2016arXiv161107493R}
A.~{Radhakrishnan}, L.~{Solus}, and C.~{Uhler}.
\newblock {Counting Markov Equivalence Classes by Number of Immoralities}.
\newblock {\em ArXiv e-prints}, November 2016.

\bibitem{2017arXiv170606091R}
A.~{Radhakrishnan}, L.~{Solus}, and C.~{Uhler}.
\newblock {Counting Markov Equivalence Classes for DAG models on Trees}.
\newblock {\em ArXiv e-prints}, June 2017.

\bibitem{Robinson1977}
R.~W. Robinson.
\newblock {\em Counting unlabeled acyclic digraphs}, pages 28--43.
\newblock Springer Berlin Heidelberg, Berlin, Heidelberg, 1977.

\bibitem{MR1997903}
Bertran Steinsky.
\newblock Enumeration of labelled chain graphs and labelled essential directed
  acyclic graphs.
\newblock {\em Discrete Math.}, 270(1-3):267--278, 2003.

\bibitem{MR3112770}
Bertran Steinsky.
\newblock Enumeration of labelled essential graphs.
\newblock {\em Ars Combin.}, 111:485--494, 2013.

\end{thebibliography}
\bibliographystyle{plain}

\end{document}